\documentclass[11pt,a4paper]{article}
\usepackage[utf8]{inputenc}
\usepackage[T2A]{fontenc}
\usepackage{amsmath}
\usepackage{amsfonts}
\usepackage{amssymb}
\usepackage{amsthm}
\usepackage[left=2cm, right=2cm, top=2cm, bottom=2.5cm]{geometry}
\usepackage{microtype}
\usepackage{graphicx}
\usepackage{float}
\usepackage{tikz-cd}
\usepackage{csquotes}
\usepackage{hyperref}

\usepackage[maxbibnames=99, style=alphabetic]{biblatex}
\author{Chebasov Dmitry}
\title{Generating functions of hypermaps and~{$m$}-constellations with two alternating boundaries}
\date{}

\newtheorem{theorem}{Theorem}
\newtheorem{lemma}[theorem]{Lemma}
\newtheorem{corollary}[theorem]{Corollary}

\theoremstyle{definition}
\newtheorem{definition}[theorem]{Definition}

\newtheorem{example}[theorem]{Example}

\theoremstyle{remark}
\newtheorem{remark}[theorem]{Remark}

\newcommand{\mat}[1]{\begin{pmatrix}#1\end{pmatrix}}
\newcommand{\varchange}[1]{\left|\begin{array}{c}#1\end{array}\right|}

\begin{document}

\maketitle

\begin{abstract}
    We study the enumeration of planar hypermaps with several alternating boundaries. We first derive a peeling recursion with an arbitrary number of alternating boundaries and, in the two-boundary case, translate it into a functional equation for the corresponding generating function. Our main result is an explicit solution of this equation for \(m\)-constellations.
    In particular, we obtain a rational expression for the genus-zero two-point function, or cylinder generating function, with two alternating boundaries.
\end{abstract}

\section{Introduction}

In this work, we consider hypermaps with boundary conditions. Geometrically speaking, a hypermap is a map (also known as a ribbon graph or a graph on a surface) with properly colored black and white faces (vertices in the dual picture). For ease of enumeration, some faces of one color can be designated as ``outer faces'' and have one of their edges marked as roots. The generating series that enumerate these rooted hypermaps are well-understood (see, for example, \cite{Kazarian_2015} for explicit expressions and an exposition of their integrability and topological recursion properties). 

We may instead consider a more general object: hypermaps with arbitrary boundary conditions. The usual hypermaps are properly colored, and so, if we designate some white faces as outer faces, all inner faces adjacent to the boundary must be black. In hypermaps with boundary condition, we assign to each boundary a word of two symbols that encodes which colors are allowed to be adjacent to the corresponding boundary edge. 

Some results on the enumeration of such objects have appeared in the literature for a while (see, for example, Chapter~8 of~\cite{eynard2016counting}). More recently, in~\cite{bouttier2021}, significant progress was made in the study of hypermaps with one boundary with alternating condition. A further restriction to~{$m$}-constellations, a particular subclass of hypermaps, provided especially nice results, including a rational parametrization of their generating function.

Since then, some further advances in the enumeration of hypermaps with one alternating boundary were made in~\cite{baillard2026} and~\cite{bouttier2026}. In this work, we advance the study of alternating boundary hypermaps with two or more borders. In particular, we give a simple recursive formula (Lemma~\ref{lem:main_recursion}) for the coefficients of their generating series, then establish a functional equation for the case of two boundaries (Theorem~\ref{thm:FE_g0n2}), and then solve it for the case of~{$m$}-constellations (Theorem~\ref{thm:a02m_solution}). In other words, we make advances in the vein of~\cite{bouttier2021} for~{$m$}-constellations with more than one boundary.

\textbf{Acknowledgements.} The author would like to thank his~advisor, Boris Bychkov, for his support throughout all stages of~this work, and Maxim Kazarian for the many valuable discussions.

\section{Hypermaps with boundary conditions}

In this section, we shall give precise definitions of our main objects, hypermaps with boundary conditions. We shall also introduce their generating functions and our main tool: a recursion that computes the generating functions of hypermaps with alternating boundary (Lemma~\ref{lem:main_recursion}).

Note that in this text we adhere to the convention where certain faces are marked and considered boundaries. In some texts on hypermaps, a dual picture is used where certain vertices are marked. Our definition closely follows the one given in Chapter~8 of~\cite{eynard2016counting}.

\begin{definition}
A \textbf{hypermap with boundary condition} (of genus {$g$} with~{$n$} roots) is the following data:
\begin{itemize}
    \item An embedded graph (map) on a surface of genus {$g$};
    \item Exactly~{$n$} of the edges of our map are marked as roots (in order). The face to the right of the root is considered an outer face. The edges of the outer face form its boundary, and the boundary is ordered starting with the root. The roots are subject to the following:
    \begin{itemize}
	    \item Outer faces are allowed to have any degree, 
	    \item Roots are numbered {$1$} to~{$n$}. 
    \end{itemize}
	\item Each boundary edge is marked with~{$\circ$} (white) or~{$\bullet$} (black). We assign to each boundary its sequence of symbols and call that sequence its type. The most notable types are:
	\begin{itemize}
        \item Alternating boundary (AB) of type {$k$}: {$\underbrace{\bullet\circ\bullet\circ\ldots \bullet\circ}_{2k}$}. 
		\item Black and white boundaries: {$\bullet\ldots\bullet$}, {$\circ\ldots \circ$}. 
		\item Mixed boundary of type {$(p, r)$}: {$\underbrace{\circ\ldots \circ}_{p}\underbrace{\bullet\circ\ldots \bullet\circ}_{2r}$}.
    \end{itemize}
    \item Each non-outer (inner) face is marked with~{$\circ$} or~{$\bullet$}. 
    \item The coloring of the map is proper. That is, a white face may be edge-adjacent to a black face or an outer face via a black boundary and vice versa.
\end{itemize}
\end{definition} 

\begin{figure}[ht]
    \centering
    \includegraphics[width=0.4\linewidth]{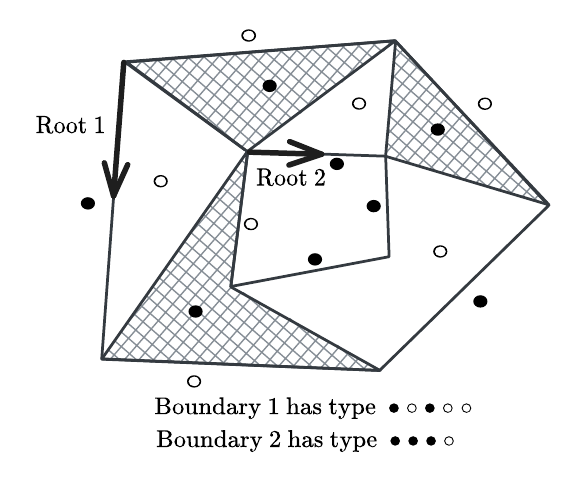}
    \caption{An example of a hypermap with~{$n=2$}, {$g=0$}.}
\end{figure}

A hypermap with boundary condition and with all white (black) monochromatic boundaries is the same as a hypermap with white (black) roots. 

Our goal is to study the generating functions that enumerate the hypermaps with all alternating boundaries (\textbf{alternating boundary hypermaps} or simply \textbf{AB hypermaps}). 

From now on, we will be working with the following classes of boundary conditions:
\begin{itemize}
    \item AB-hypermaps of signature {$\bar r=(r_1, \ldots, r_n)$}: all boundaries are alternating of type {$r_i$}. The letter {$A$} will be used for their generating functions.
    \item First-mixed hypermaps of signature {$(p, \bar{r}) = (p, (r_1, r_2, \ldots, r_n))$}: the first boundary is of mixed type {$(p, r_1)$} and the rest are alternating type {$r_i$}. The letter {$M$} will be used for their generating functions.
    \item Pure-first-mixed hypermaps of signature {$(p, \bar r) = (p, (r_2, \ldots, r_n))$}: the first boundary is white of type {$p$} and the rest are alternating type {$r_i$}. In other words, they are first-mixed of signature {$(p, (0, r_2,\ldots, r_n))$}. The letter {$P$} will be used for their generating functions.
    \item White boundary hypermaps, or simply hypermaps. The letter {$H$} will be used for their generating functions. When~{$n=1$}, note that~{$P^{(0,1)}_{p} = H^{(0,1)}_p$} by definition.
\end{itemize}

For all of the above classes of hypermaps, we introduce their generating functions. First, we introduce the series
\begin{equation}
    M^{(g,n)}_{p, \bar r} = \sum_{\Gamma} t^v \prod_{i\geq 1} t^{l_i} \tilde t^{\tilde l_i}.
\end{equation}
Here~{$\Gamma$} runs over all first-mixed hypermaps of genus {$g$}, with~{$n$} boundaries and of signature {$(p, \bar r)$}. The numbers {$v,l_i, \tilde l_i$} are the vertex count, the black inner face of degree {$i$} count and the white inner face of degree {$i$} count of~{$\Gamma$}, respectively.

These series account for all types of hypermaps we use in this work. From them we construct the big generating series for first-mixed hypermaps:
\begin{gather}\label{eq:m_def}
    M^{(g,n)}(w_1, \ldots, w_n, x) = \sum_{p+r_1\geq 1, r_2, \ldots, r_n \geq 1} \frac{\prod_{i=1}^n w_i^{r_i}}{x^{p+1}}M_{p,\bar r}^{(g,n)}, \\
    M^{(0,1)} (w,x)= \sum_{p,r \geq 0} w^r x^{-1-p} M^{(0,1)}_{p,(r)}. 
\end{gather}

In \eqref{eq:m_def} the variables $p, r_1$ range from $0$, the same holds in the rest of the text.

We define~{$A^{(g,n)}(w_1, \ldots, w_n), P^{(g,n)}(w_2,\ldots, w_n,x), H^{(g,n)}(x_1, \ldots, x_n)$} in a similar way with the same meanings of~{$w_i,x$}:
\begin{gather}
    A^{(g,n)}(w_1, \ldots, w_n)= \sum_{r_1, \ldots, r_n \geq 1} \prod_{i=1}^n w_i^{r_i}M_{0,\bar r}^{(g,n)}, \\
    P^{(g,n)}(w_2, \ldots, w_n, x) = \sum_{p, r_2, \ldots, r_n \geq 1} \frac{\prod_{i=2}^n w_i^{r_i}}{x^{p+1}}M_{p,(0, r_2, \ldots)}^{(g,n)}, \\
    H^{(g,1)}(x)= \sum_{p \geq 1} \frac{1}{x^{p+1}}M_{p,0}^{(g,1)}, \\
    A^{(0,1)}(w)=\sum_{r\geq 0} w^r M^{(0,1)}_{0,(r)},\quad H^{(0,1)}(x) = \sum_{p\geq 0} x^{-1-p} M^{(0,1)}_{p,(0)}.
\end{gather}

\begin{example}
    The series {$M^{(1,1)}_{2,(2)}$} counts hypermaps on a torus with one boundary of mixed type {$(2, 2)$}. That is, their boundary is colored like this: {$(\circ\circ\bullet\circ\bullet\circ)$}. It appears as a slice of~{$M^{(1,1)}$}:
\begin{equation}
    M^{(1,1)}(w_1, x) = M^{(1,1)}_{1,(0)}\frac{1}{x^{2}} + \ldots + M^{(1,1)}_{2,(2)}\frac{w_1^2}{x^{3}}+ \ldots.
\end{equation}

\end{example}

\begin{example}
    We shall later encounter the sum 
    \begin{equation}
    \sum_{p\geq 1, r_1\geq 0, r_2\geq 1} M^{(0,2)}_{p, (r_1, r_2)} w_1^{r_1} w_2^{r_2} x^{-1-p}.
    \end{equation}
    It differs from~{$M^{(0,2)}$} since it excludes the hypermaps with the first alternating boundary, so it is equal to~{$M^{(0,2)} - \frac{A^{(0,2)}}{x}$}.
\end{example}

We shall count a hypermap consisting of a single vertex and a single face as a hypermap with empty signature. It is the only contribution to
\begin{equation}
M^{(0,1)}_{0,(0)} = H^{(0,1)}_{0} = A^{(0,1)}_{0} = t.
\end{equation}

These conventions are based on those used in~\cite{bouttier2021}. Lemma~\ref{lem:main_recursion} below is a generalization of Lemma~1 in~\cite{bouttier2021} to the case of multiple boundaries. Note, however, that our~{$A^{(0,1)}$} is shifted by~{$w^{-1}$} compared to~\cite{bouttier2021}.

Note that~{$A^{(g,n)}$} are obviously symmetric with respect to~{$w_i$}.

Now we introduce two operations, peeling and white boundary erasing, that will allow us to construct a recursion reducing the computation of~{$M^{(0,n)}$} (and, as a consequence, {$A^{(0,n)}$}) to the computation of~{$H^{(0,1)}$}, which is well known.

\begin{definition}
    The \textbf{peeling at the first black edge} (or simply \textbf{peeling}) is an operation defined on hypermaps with at least one black edge on its first boundary. It assigns to a given hypermap {$\Gamma$} a hypermap that is the same as~{$\Gamma$} but with the first {$\bullet$} edge along its first boundary removed.

    The \textbf{white boundary erasure} is an operation defined on pure-first-mixed hypermaps. It assigns to a hypermap of this type the same hypermap, but with the first outer face replaced with an inner white face and the numbers of the rest of the outer faces shifted down by one. It maps hypermaps counted by~{$P^{(g,n)}_{p,\bar r}$} to those counted by~{$A^{(g,n-1)}_{\bar r}$}.  
\end{definition}

\section{Genus 0}

\begin{lemma}[Peeling of mixed boundary hypermaps]\label{lem:main_recursion}
    Peeling defines a bijection that results in an equation:
    \begin{align}
        M^{(0,n)}_{p,\bar r} = &\sum_{L\geq 1}t_L M^{(0,n)}_{p+L, (r_1-1, r_2, \ldots, r_n)} + \\
        +& \sum_{j = 1}^{p}\sum_{I\sqcup J= [[2,n]]}M^{(0,|I|+1)}_{j, (r_1-1,\pi_I(\bar r))} M^{(0,|J|+1)}_{p-j, (0,\pi_{J}(\bar r))} + \notag \\
        +& \sum_{s = 1}^{r_1}\sum_{I\sqcup J= [[2,n]]}M^{(0,|I|+1)}_{p, (r_1-s,\pi_I(\bar r))} M^{(0,|J|+1)}_{0, (s-1,\pi_{J}(\bar r))} + \notag \\
        +& \sum_{k=2}^n r_k\cdot  M^{(0,n-1)}_{p,(r_1 + r_k-1, r_2, \ldots, \widehat{r_k}, \ldots, r_n)}, \notag
    \end{align}

    where: 
    \begin{itemize}
        \item {$[[2,n]] = \{2,3,\ldots, n\}$};
        \item {$\pi_I : \mathbb{Z}_{\geq0}^{n} \to \mathbb{Z}_{\geq 0}^{|I|}$} is the projection onto basis vectors indexed by~{$I$}. E.g. {$\pi_{\{1,3\}}(a,b,c)=(a,c)$}. 
    \end{itemize}

    To account for the case where one of the boundaries is of length zero, we should treat~{$M^{(0,|J|+1)}_{0,(0, \pi_J(\bar r))}$} as zero in cases where~{$\pi_J(\bar r) \neq \emptyset$}. 
\end{lemma}

\begin{proof}
    Each line of the formula corresponds to one of the four possible cases of what can be behind the peeled edge:
    \begin{enumerate}
        \item A black face of degree {$L$}. 
        \item The first boundary, and the other side has already been traversed.
        \item The first boundary, and the other side has not been traversed.
        \item A~{$k$}-th boundary.
    \end{enumerate}
    Cases 1 and 4 are straightforward: the revealed face is merged into the first boundary. Cases 2 and 3, where our hypermap splits into two, are handled in Figure~\ref{fig:peeling}.
    \begin{figure}[ht]
    \centering
    \includegraphics[width=\linewidth]{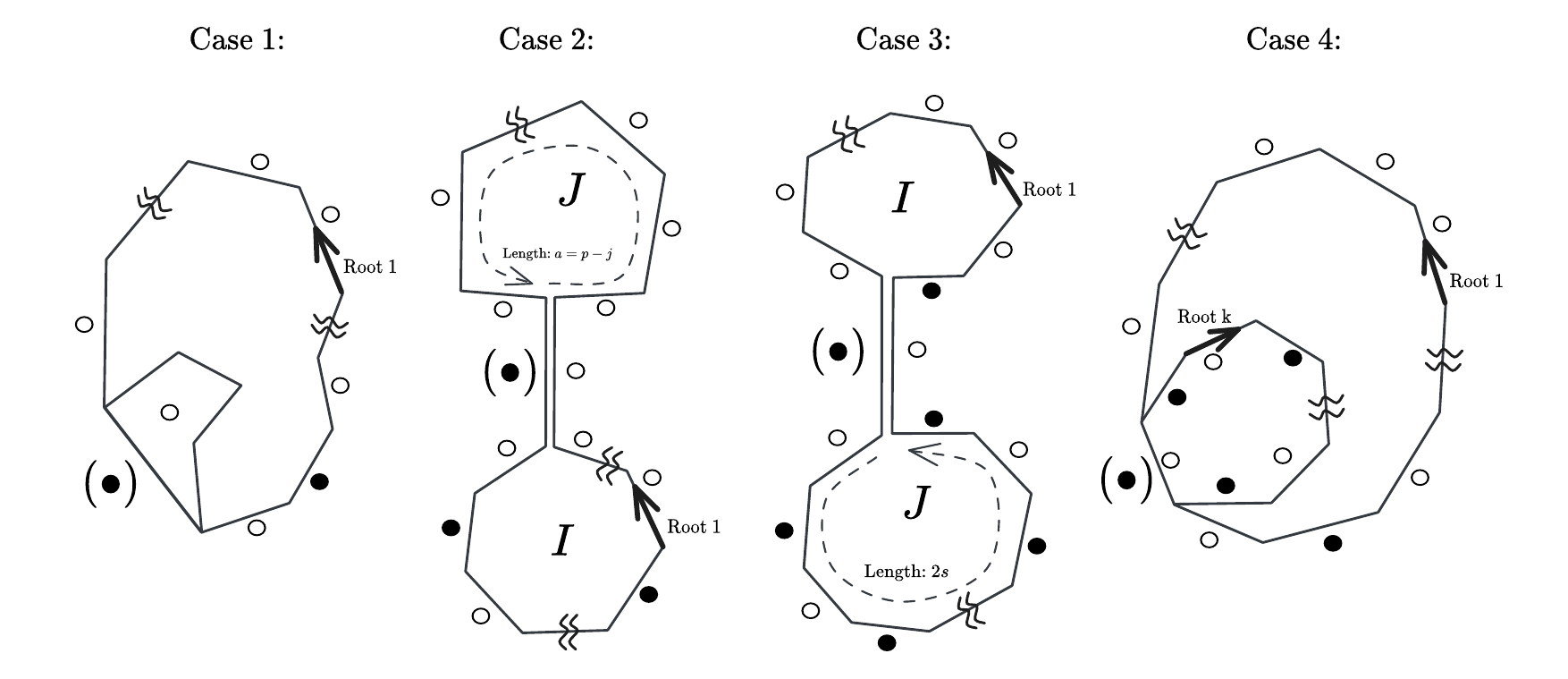}
    \caption{Peeling the first {$\bullet$}.}
    \label{fig:peeling}
    \end{figure}
\end{proof}

\begin{remark}\label{rem:genus}
    In the case of~{$g\geq 1$}, a similar recursion holds. Only in the case of higher genus should an extra term be introduced. It corresponds to the case where the peeled edge makes a full circle around a handle. 
\end{remark}

\begin{example}
    By setting~{$g=0$} and~{$n=1$}, we recover
    \begin{align}
        M^{(0,1)}_{p, r} = &\sum_{L\geq 1}t_L M^{(0,1)}_{p+L, r - 1} + \\
        +& \sum_{j = 1}^{p}M^{(0,1)}_{j, (r_1-1)} M^{(0,1)}_{p-j, (0)} + \sum_{s = 1}^{r_1}M^{(0,1)}_{p, (r_1-s)} M^{(0,1)}_{0, (s-1)}.\notag
    \end{align}
    Up to a change of summation variables this is the peeling recursion of~\cite{bouttier2021}.
\end{example}

\begin{lemma}[White boundary erasure of pure-first-mixed hypermaps]\label{lem:erase}
    Erasing the first boundary of a pure-first-mixed boundary hypermap results in the following equation for~{$n>1$}:
    \begin{equation}
        P^{(0,n)}_{p, (r_1,\ldots, r_{n-1})}= p\cdot\partial_{t_p}  A^{(0,n-1)}_{(r_1, \ldots, r_{n-1})}.
    \end{equation} 
\end{lemma}

\begin{proof}
    Let us do the erasure in reverse. Consider a hypermap with~{$n-1$} alternating boundaries, black boundary weight {$\prod_i t_i^{l_i}$}, and signature {$(r_1, \ldots, r_{n-1})$}. We want to choose the root on a black face of degree {$p$} that will make it into a new boundary. There are~{$l_p$} such faces with~{$p$} edges each that could serve as roots. The degree of~{$t_p$} decreases by 1 after choosing the root.
\end{proof}  

\begin{remark}
    A similar result clearly holds for the generating function of hypermaps with two or more monochromatic boundaries. This will be used later.
\end{remark}

After peeling, the number of~{$\bullet$} edges decreases, and after erasure, the number of boundaries decreases. One of these operations is always applicable to a given class of mixed boundary hypermaps other than the ones with one white boundary. Thus, by applying them repeatedly, we eventually express any~{$M^{(0,n)}_{(\ldots)}$} in terms of~{$H^{(0,1)}_{p}$}. 

This principle clearly generalizes to~{$g>0$} (see Remark~\ref{rem:genus}), but we omit the discussion for brevity.

\section{Specialization to hypermaps with two alternating boundaries}

\begin{corollary}\label{col:ab_general_g0n2}
    Setting~{$n=2$}, we get a recursive formula
    \begin{align} \label{eq:ab_general_g0n2}
        M^{(0,2)}_{p,(r_1, r_2)} = &\sum_{L\geq 1}t_L M^{(0,2)}_{p+L, (r_1-1, r_2)} + \\
        +& \sum_{j = 1}^{p}\left(M^{(0,2)}_{j, (r_1-1,r_2)} M^{(0,1)}_{p-j, (0)} + M^{(0,1)}_{j, (r_1-1)} M^{(0,2)}_{p-j, (0,r_2)}\right)+ \notag \\
        +& \sum_{s = 1}^{r_1}\left(M^{(0,2)}_{p, (r_1-s,r_2)} M^{(0,1)}_{0, (s-1)} + M^{(0,1)}_{p, (r_1-s)} M^{(0,2)}_{0, (s-1,r_2)}\right)+ \notag\\
        +& r_2\cdot  M^{(0,1)}_{p,(r_1 + r_2-1)}, \notag
    \end{align}
    where the degenerate terms are skipped: {$M^{(0,2)}_{0,(0,r_2)}=0$} for~{$r_2\geq1$}, so in the second sum only~{$j\leq p-1$} contributes and in the third sum only~{$s\geq 2$}, while~{$M^{(0,1)}_{0,(0)}=t$}, so that~{$j=p$} and~{$s=r_1$} do contribute.
\end{corollary}

In~\cite{bouttier2021}, a functional equation is introduced for~{$M^{(0,1)}$}. Our recursion allows us to obtain functional equations for higher generating functions. As a corollary, we shall provide the equation for~{$M^{(0,2)}$} and use it to obtain an equation for~{$A^{(0,2)}$}. In a later section, we shall provide a rational solution of this equation when specialized to a particular class of AB hypermaps: alternating boundary {$m$}-constellations. 

\begin{theorem}\label{thm:FE_g0n2}
    The big generating function {$M^{(0,2)}$} satisfies a functional equation
    \begin{align}
    M^{(0,2)} - P^{(0,2)} &= \\
    &=w_1\left[ T(x)M^{(0,2)} \right]_{x^{<0}}+ \tag{Term a} \\
    &+w_1xH^{(0,1)}\left(M^{(0,2)} - \frac{A^{(0,2)}}{x}\right)+ \tag{Term b} \\
    &+w_1xP^{(0,2)}\left(M^{(0,1)} - \frac{A^{(0,1)}}{x}\right)+ \tag{Term c} \\
    &+w_1 A^{(0,1)}(w_1)M^{(0,2)}+ \tag{Term d} \\
    &+w_1 A^{(0,2)}M^{(0,1)}+ \tag{Term e} \\
    &+\frac{w_1^2 w_2}{(w_1-w_2)^2}\,
            \left(M^{(0,1)}(x,w_1)-M^{(0,1)}(x,w_2)\right)
        - \frac{w_1 w_2^{\,2}}{w_1-w_2}\,
            \frac{\partial M^{(0,1)}}{\partial w_2}(x,w_2), \tag{Term f}
    \end{align}
    where
\begin{equation}
T(x) = \sum_{i\geq 1} t_i x^{i},
\end{equation}
all the generating functions are evaluated at the arguments indicated by their superscripts, and the square brackets denote the extraction of coefficients (in this case, from terms in~{$x$} in negative degrees).
\end{theorem}

\begin{proof}
We will go term by term over the right-hand side and establish equivalence to the six terms of the recursion~\eqref{eq:ab_general_g0n2}. Note that the summation in~{$M^{(0,2)}$} goes over all non-negative {$p,r_1, r_2$} such that~{$p+r_1 \geq 1, r_2 \geq 1$}, since we do not have boundaries of length zero. So the left-hand side of the equation becomes the sum of~{$M^{(0,2)}_{p,(r_1, r_2)}$} where~{$p\geq 0,r_1\geq 1, r_2 \geq 1$}, that is~{$M^{(0,2)}-P^{(0,2)}$} (the terms with~{$r_1=0$} are the pure-first-mixed ones counted by~{$P^{(0,2)}$}).

In a similar vein, we shall use the following identities:
\begin{gather} \label{eq:sum_correction1}
\sum_{\substack{p\geq0, r_1,r_2\geq 1}}\frac{w_1^{r_1}w_2^{r_2}}{x^{p+1}}M^{(0,2)}_{p,(r_1,r_2)} = M^{(0,2)}-P^{(0,2)}, \quad
\sum_{\substack{r_1,r_2\geq 1}}\frac{w_1^{r_1}w_2^{r_2}}{x}M^{(0,2)}_{0,(r_1,r_2)} = \frac{A^{(0,2)}}{x},
\end{gather}
\begin{gather} \label{eq:sum_correction2}
\sum_{l\geq0}H^{(0,1)}_{l}x^{-l} = xH^{(0,1)}(x),\quad
\sum_{\substack{p\geq0, r\geq0}}w_1^{r}x^{-p-1}M^{(0,1)}_{p,(r)} = M^{(0,1)}(w_1,x).
\end{gather}

Since the right-hand side is exactly
\begin{equation}
\sum_{\begin{array}{c}
    p \geq 0 \\ r_1 \geq 1, r_2 \geq 1
\end{array}}
M_{p,(r_1,r_2)}^{(0,2)} w_1^{r_1} w_2^{r_2} x^{-1-p},
\end{equation}
the recursion~\eqref{eq:ab_general_g0n2} is applicable to each summand. Below we recover each term of the theorem from the corresponding term of the recursion.

\textbf{Term a.} 
\begin{gather}
\sum_{\begin{array}{c}
    p \geq 0 \\ r_1 \geq 1, r_2 \geq 1
\end{array}}
\sum_{L\geq 1}t_L M^{(0,2)}_{p+L, (r_1-1, r_2)} w_1^{r_1} w_2^{r_2} x^{-1-p} = \varchange{p=p'-L \\ r_1'=r_1 - 1}= \\
=\sum_{\begin{array}{c}
    p' \geq L \\ r_1' \geq 0, r_2 \geq 1
\end{array}}
\sum_{L\geq 1}t_L M^{(0,2)}_{p', (r_1', r_2)} w_1^{r_1'+1} w_2^{r_2} x^{-1-p'+L} = \notag \\
= w_1 \sum_{L\geq 1} x^L t_L  \sum_{\begin{array}{c}
    p' \geq L \\ r_1' \geq 0, r_2 \geq 1
\end{array}}
M_{p',(r_1', r_2)} w_1^{r_1'}w_2^{r_2}x^{-1-p'}= w_1 \left[T(x)M^{(0,2)}\right]_{x^{<0}}. \notag
\end{gather}

\textbf{Term b.} 
\begin{gather}
\sum_{\begin{array}{c}
    p \geq 0 \\ r_1 \geq 1, r_2 \geq 1
\end{array}} \sum_{j=1}^{p} M_{j, (r_1-1, r_2)}^{(0,2)} M_{p-j,(0)}^{(0,1)} w_1^{r_1} w_2^{r_2} x^{-1-p}= \varchange{q=j\\ r_1'=r_1-1}=\\
=\sum_{\begin{array}{c}
    q \geq 1, l\geq 0 \\ r_1' \geq 0, r_2 \geq 1
\end{array}} M^{(0,2)}_{q,(r_1', r_2)} H^{(0,1)}_{l} w_1^{r_1'+1} w_2^{r_2} x^{-(q+l+1)}=  \notag \\
=w_1\left(
\sum_{\begin{array}{c}
    q \geq 1 \\ r_1' \geq 0, r_2 \geq 1
\end{array}} M^{(0,2)}_{q,(r_1', r_2)} w_1^{r_1'} w_2^{r_2} x^{-1-q}
\right)
\left(
\sum_{l\geq 0} H^{(0,1)}_{l} x^{-l}
\right) = \notag \\
= w_1xH^{(0,1)}\left(M^{(0,2)}-\frac{A^{(0,2)}}{x}\right), \notag
\end{gather}
where we used~\eqref{eq:sum_correction2} and, for the first factor, \eqref{eq:sum_correction1} together with the fact that the terms with~{$q=0$} sum up to~{$A^{(0,2)}/x$}.

\textbf{Term c.} Similarly, with~{$q=j$} and~{$l=p-j$},
\begin{gather}
\sum_{\begin{array}{c}
    p \geq 0 \\ r_1 \geq 1, r_2 \geq 1
\end{array}} \sum_{j=1}^{p-1} M_{j, (r_1-1)}^{(0,1)} M_{p-j,(0,r_2)}^{(0,2)} w_1^{r_1} w_2^{r_2} x^{-1-p}= \\
=\sum_{\begin{array}{c}
    q \geq 1, l\geq 1 \\ r_1' \geq 0, r_2 \geq 1
\end{array}} M^{(0,1)}_{q,(r_1')} P^{(0,2)}_{l,(r_2)} w_1^{r_1'+1} w_2^{r_2} x^{-(q+l+1)}= \notag  \\
=w_1\left(
\sum_{\begin{array}{c}
    q \geq 1 \\ r_1' \geq 0
\end{array}} M^{(0,1)}_{q,(r_1')} w_1^{r_1'} x^{-1-q}
\right)
\left(
\sum_{l\geq 1, r_2 \geq 1} P^{(0,2)}_{l,(r_2)} w_2^{r_2}x^{-l}
\right) = \notag  \\
= w_1xP^{(0,2)}\left(M^{(0,1)}-\frac{A^{(0,1)}}{x}\right), \notag
\end{gather}
where we used~{$M^{(0,2)}_{q,(0,r_2)}=P^{(0,2)}_{q,(r_2)}$} and~\eqref{eq:sum_correction2}.

\textbf{Term d.} 
\begin{gather}
\sum_{\begin{array}{c}
    p \geq 0 \\ r_1 \geq 1, r_2 \geq 1
\end{array}} \sum_{s=1}^{r_1} M_{p, (r_1-s, r_2)}^{(0,2)} M_{0,(s-1)}^{(0,1)} w_1^{r_1} w_2^{r_2} x^{-1-p}=\varchange{r'_1 = r_1 - s \\ r_1'\geq 0,\ b=s-1 \geq 0} = \\
\label{pf:term_d} =\sum_{\begin{array}{c}
    p \geq 0, b\geq 0 \\ r_1' \geq 0, r_2 \geq 1 
\end{array}}
M_{p, (r_1', r_2)}^{(0,2)} A^{(0,1)}_{b} w_1^{r_1'+b+1} w_2^{r_2} x^{-1-p}. \notag\end{gather}
Note that the terms with $p=0, r_1'=0$ are zero, so we may omit them from the sum, and \eqref{pf:term_d} can now be regrouped as:
\begin{gather}
\left(
    \sum_{\begin{array}{c}
    p +r_1' \geq 1, r_2 \geq 1
\end{array}} M^{(0,2)}_{p,(r_1', r_2)} w_1^{r_1'}w_2^{r_2} x^{-1-p} 
\right) 
\left(
    \sum_{b\geq 0} A^{(0,1)}_{b}w_1^{b+1}
\right) = \\ 
= w_1 A^{(0,1)}(w_1)M^{(0,2)}. \notag
\end{gather}

\textbf{Term e.} Similar to \textbf{Term d}: with~{$r_1'=r_1-s\geq0$}, {$b=s-1\geq1$},
\begin{gather}
\sum_{\begin{array}{c}
    p \geq 0 \\ r_1 \geq 1, r_2 \geq 1
\end{array}} \sum_{s = 2}^{r_1} M_{p, (r_1-s)}^{(0,1)} M_{0,(s-1,r_2)}^{(0,2)} w_1^{r_1} w_2^{r_2} x^{-1-p}= \\
=\sum_{\begin{array}{c}
    p \geq 0, b\geq 1 \\ r_1' \geq 0, r_2 \geq 1 
\end{array}}
M_{p, (r_1')}^{(0,1)} A^{(0,2)}_{(b,r_2)} w_1^{r_1'+b+1} w_2^{r_2} x^{-1-p}=\notag  \\
=\left(
    \sum_{p\geq 0, r_1'\geq 0} M_{p, (r_1')}^{(0,1)} w_1^{r_1'} x^{-1-p} 
\right) 
\left(
    \sum_{r_2 \geq 1,b\geq 1} A^{(0,2)}_{(b,r_2)}w_1^{b+1}w_2^{r_2}
\right) =\notag  \\ 
= w_1A^{(0,2)}M^{(0,1)}. \notag
\end{gather}

\textbf{Term f.} Starting from the sixth term of the recursion and setting~{$p'=r_1 + r_2  - 1 \geq 1$} we get:
\begin{gather}
    \sum_{\begin{array}{c}
    p \geq 0 \\ r_1 \geq 1, r_2 \geq 1
\end{array}} r_2 M_{p, (r_1+ r_2 - 1)}^{(0,1)} w_1^{r_1}w_2^{r_2} x^{-1-p} = \\
    = \sum_{p\geq 0, p'\geq 1} M_{p,(p')}^{(0,1)} x^{-1-p} \sum_{s=1}^{p'} (p'-s+1) w_1^s w_2^{p' - s + 1}. \notag
\end{gather}
Making use of the identity
\begin{equation}
    \sum_{s=1}^{p'} (p'-s+1) w_1^s w_2^{p'-s+1} = \frac{w_1^2 w_2}{(w_1 - w_2)^2}(w^{p'}_1 - w_2^{p'}) - \frac{w_1 w_2^2}{w_1 - w_2} p' w^{p'-1},
\end{equation}
we recover 
\begin{equation}
\mathrm{Term\ f} = \frac{w_1^2 w_2}{(w_1 - w_2)^2}\left(\mathcal{M}(x,w_1) - \mathcal{M}(x,w_2)\right) - \frac{w_1 w_2^2}{w_1 - w_2} \frac{\partial \mathcal{M}(x, w_2)}{\partial w_2},
\end{equation}
where~{$\mathcal{M} = M^{(0,1)} - H^{(0,1)}$}. But since~{$H^{(0,1)}$} does not depend on~{$w_i$}, we may just replace~{$\mathcal{M}$} by~{$M^{(0,1)}$} in the equality above.
\end{proof}

\section{Specialization to~\texorpdfstring{$m$}{m}-constellations}

\begin{definition}
    An~{$m$}-constellation with mixed boundary is a hypermap with mixed boundary that only has inner black faces of degree {$m$} and inner white faces of degree divisible by~{$m$}. We will denote the generating functions of~{$m$}-constellations by~{$X^{(g,n),m}$}, where~{$X$} can be one of the same letters as before. 
\end{definition}

The main result of~\cite{bouttier2021} is a rational expression for~{$A^{(0,1),m}$} obtained by solving the corresponding functional equation. We obtain~{$A^{(0,2),m}$} in a similar way.

The series {$M^{(0,2),m}$} are obtained from~{$M^{(0,2)}$} by substituting:
\begin{equation}
    \tilde{t}_i\mapsto 
\begin{cases}
    x_{i/m}\quad m|i \\
    0\text{ otherwise}.
\end{cases}\quad t_i\mapsto \delta_{i,m}.
\end{equation}

\begin{lemma}[Divisibility]\label{lem:divisibility}
    In an~{$m$}-constellation, the length of the monochromatic part of the word of a boundary is divisible by~{$m$}. In particular,
    \begin{equation}
    M^{(0,n),m}_{p,\bar r} = P^{(0,n),m}_{p,\bar r} = 0 \quad\text{unless } m\mid p,
    \end{equation}
    and consequently, writing {$A^{(0,2),m}=A^{(0,2),m}(w_1,w_2)$} and~{$A^{(0,1),m}=A^{(0,1),m}(w_1)$},
    \begin{gather}
    M^{(0,2),m}=\frac{A^{(0,2),m}}{x}+O(x^{-m-1}),\quad
    M^{(0,1),m}=\frac{A^{(0,1),m}}{x}+O(x^{-m-1}),\\
    H^{(0,1),m}=\frac{t}{x}+O(x^{-m-1}),\quad
    P^{(0,2),m}=O(x^{-m-1}). \notag
    \end{gather}
\end{lemma}

\begin{proof}
    For a simple closed cycle of the map, let~{$a$} (resp.\ {$b$}) be the number of its edges traversed with (resp.\ against) this orientation. We claim that~{$m\mid a-b$}. Indeed, the claim is additive when an edge lying inside the cycle splits it into two cycles, so we just need to check it for a cycle bounding a single face {$F$}: then the cycle is the border of~{$F$} traversed so that~{$F$} lies on the right, so all its edges are counted in~{$a$} if~{$F$} is white and in~{$b$} if~{$F$} is black; hence~{$a-b=\deg F$} or~{$a-b=-\deg F$}, respectively. Since white faces have degree divisible by~{$m$} and black ones have degree exactly~{$m$}, we get what we want by induction.

    The word {$\circ^p(\bullet\circ)^{r}$} contains {$r$} occurrences of~{$\bullet$} and~{$p+r$} occurrences of~{$\circ$}, so the signed count is~{$r-(p+r)=-p$}. Hence~{$m\mid p$}.

    The asymptotics follow since maps counted by~{$M^{(0,2)}_{p,(r_1,r_2)}$} have~{$p\equiv0\bmod m$}, so that the only power of~{$x$} of degree {$>-m-1$} occurring in~{$M^{(0,2),m}$} is~{$x^{-1}$}, whose coefficient is~{$\sum_{r_1,r_2\geq1}A^{(0,2)}_{(r_1,r_2)}w_1^{r_1}w_2^{r_2}=A^{(0,2),m}$}. The other three series are treated in the same way.
\end{proof}

Now we are ready to~formulate and prove the main result of the paper: a~rational parametrization of~$A^{(0,2),m}$.

\begin{theorem}\label{thm:a02m_solution}
    Fix a positive integer {$d$}, and from now on we will consider only~{$m$}-constellations with white faces of degree not exceeding {$md$}. In other words, {$x_i = 0$} for~{$i>d$}.

    The generating function of~{$m$}-constellations with two alternating boundaries of genus {$0$} is equal to:
    \begin{equation}\label{eq:a02m_solution}
        A^{(0,2),m} = \frac{w_1 w_2}{w_1'w_2'}\left(\frac{m}{(s_1 - s_2)^2} - (m-1)\frac{\xi_1' \xi_2'}{(\xi_1 - \xi_2)^2}\right) - \frac{w_1 w_2}{(w_1 - w_2)^2},
    \end{equation} 
    where the symbols are determined from the following relations:
    \begin{gather}
        V = t + \sum_{i = 1}^d \mat{mi-1\\i} x_i V^{(m-1)i} \\
        \alpha_k = \sum_{i=k}^d \mat{mi-1\\ i-k} x_i V^{(m-1)i+k-1}, \quad k\in\{1, \ldots, d\} \\
        w(s) = s \frac{(1+\sum_{k=1}^d \alpha_k s^k)^{m-2}}{(1+V s)^2}\\
        x(z) = z + \sum_{k=1}^d \alpha_k z^{1-mk}\quad   y(z) = \frac{V}{z} + z^{m-1}, \quad s = z^{-m} \\
        \chi(s) = (x(z(s)))^m\quad \xi(s) = x(z(s))\,y(z(s)),\quad w(s) = \frac{\chi(s)}{\xi(s)^2},
    \end{gather}
    and~{$w_i = w(s_i)$}, {$\xi_i = \xi(s_i)$}, while~{$f'$} denotes a derivative in the~{$s$} variable. The right-hand side of~\eqref{eq:a02m_solution} is understood as a series in~{$w_1, w_2$}: the three terms are first collected as a series in~{$s_1, s_2$} (the poles on the diagonal cancel), and then~{$w_i=w(s_i)$} is substituted.
\end{theorem}

\begin{remark}
    The above definitions are not ad hoc. They come from the spectral curve topological recursion for the enumeration of hypermaps. The machinery of topological recursion is not used in the proof and is not required for understanding. The rest of the remark assumes the reader's familiarity with TR. 
    
    By examining the result of~\cite{bouttier2021} (later stated explicitly in~\cite{baillard2026}), {$A^{(0,1),m}$} can be interpreted in terms of a~particular topological recursion's spectral curve. To elaborate: define~{$\bar \Sigma$} as a pullback {$\bar \Sigma \to_{\chi(s)} \Sigma$} where~{$\Sigma$} is a spectral curve of the hypermap enumeration problem (see Chapter~8 of~\cite{eynard2016counting}, \cite{Kazarian_2015} or the addendum of~\cite{bouttier2021}). This spectral curve is given by~{$x(z), y(z)$} in our notation. The rational form of~{$A^{(0,1),m}$} can then be easily interpreted as a function on~{$\bar\Sigma$}.

    Our solution for~{$A^{(0,2),m}$} is also given in terms of this cover of the hypermap spectral curve, and it is similar to the pullback of the standard {$B(z_1, z_2)$} data of the corresponding TR. Interpretation of higher {$A^{(g,n),m}$} in terms of this curve is an avenue for future research. 
\end{remark}

\subsection{Proof of Theorem \ref{thm:a02m_solution}} 
\subsubsection{Enumeration of hypermaps with monochromatic boundary}

We shall need a few facts from the topological recursion for monochromatic boundary hypermaps (see Chapter~8 of~\cite{eynard2016counting} or~\cite{Kazarian_2015}). For our purposes we need the fact that, for certain $x(z), y(z)$, after setting $Y(x(z)))=y(z)$, the following two equalities hold: 
\begin{gather}
    Y(x) = \sum_{i=2}^d t_i x^{i-1} + H^{(0,1)}(x), \\
    H^{(0,2)}(x_1,x_2)dx_1 dx_2 = \frac{dz_1 dz_2}{(z_1 - z_2)^2}-\frac{dx_1 dx_2}{(x_1 - x_2)^2}. \label{eq:bergman}
\end{gather}
where $z_i = z(x_i)$. We refer the reader to the cited works for details. The only thing important for our proof is that after specializing to $m$-constellations $x(z), y(z)$ become precisely those defined in Theorem~\ref{thm:a02m_solution}. Note that these results hold when, as we do, the degree of faces is bounded.

Let us, following the same variable and weights conventions as before, define $\mathcal{H}(x_1, x_2)$ to be the generating function counting hypermaps with two monochromatic boundaries where the second boundary has a length divisible by $m$. 

We introduce an operator $L_X$ acting on series $F$:
\begin{equation}
L_XF = \sum_{k\geq 1} mk X^{-1-mk} \frac{\partial F}{\partial t_{mk}}.
\end{equation}
This operator closely resembles the one used in our Lemma~\ref{lem:erase}. Following the same logic as the proof of that lemma, we see that
\begin{equation} \label{eq:LH}
    \mathcal{H}(x_1, x_2) = L_{x_2} H^{(0,1)}(x_1).
\end{equation}
Similarly
\begin{equation}\label{eq:LA}
    L_{x_2}A^{(0,1)}(w) = P^{(0,2)}(x_2, w).
\end{equation}

\subsubsection{Specializing to~\texorpdfstring{$m$}{m}-constellations}

From now on, we have specialized all of the generating functions to the case of $m$-constellations. We shall drop the superscript $m$ for visual clarity, e.g. $M^{(0,2)}$ from now on denotes $M^{(0,2),m}$ unless stated otherwise.

The series $x(z)$ and $y(z)$ from now on will have the definitions as in the theorem statement:
\begin{equation}
x(z) = z + \sum_{k=1}^d \alpha_k z^{1-mk},\quad   y(z) = \frac{V}{z} + z^{m-1}.
\end{equation}

From the proof of Lemma~2 of \cite{bouttier2021} we know the~${w = w(x)}$ from the theorem is such that ${K(w(x), x)=0}$, where
\begin{equation}\label{eq:on_a_curve}
    K(w,x) = 1 - wA^{(0,1)}(w) - wxY(x).
\end{equation}
This will be used in some of the proofs, somtimes implicitly, without writing that~$w$ is treated as~$w(x)$. One may instead imagine that we are working on a curve given by~${K(w,x)=0}$. 

We shall start with a few observations about $x(z)$ and $\mathcal{H}$.

\setlength{\parindent}{0pt}

\begin{lemma}\label{lem:step1}
    The following equalities hold:
    \begin{enumerate}
        \item Let $\zeta^m = 1$. Then $x(\zeta z) = \zeta x(z)$.
        \item $\frac{1}{m}\sum_{\zeta^m=1} \zeta \cdot (\zeta x)^{-1-q} = \begin{cases}
            x^{-1-q}&m\mid q \\ 0 & m\not| q.\end{cases}$
        \item $\mathcal{H}(x_1, x_2) = \frac{1}{m} \sum_{\zeta^m=1} \zeta H^{(0,2)}(x_1, \zeta x_2)$.
        \item $\mathcal{H}(x_1, x_2) dx_1 dx_2 = \frac{1}{m}(B_s - B_{\chi})$.
    \end{enumerate}
    In the last item we set
    \begin{equation}
    B_{f(s)} = \frac{df(s_1) df(s_2)}{(f(s_1)- f(s_2))^2},
    \end{equation}
    the $s$ and $\chi$ are the same as in the statement of the Theorem~\ref{thm:a02m_solution}, and $x_i = x(z(s_i))$. 
\end{lemma}

\begin{proof}\hfill
\begin{enumerate}
    \item \begin{gather} \label{pf:1}
    x(\zeta z)= \zeta z + \sum \alpha_k \zeta^{1-mk}z^{1-mk} = \zeta \left(z+ \sum \alpha_k \zeta^{-mk} z^{1-mk}\right) = \zeta x(z).
    \end{gather}
    \item \begin{gather} \label{pf:2}
        \frac{1}{m}\sum_{\zeta^m=1} \zeta(\zeta(x_2))^{-q-1}=x_2^{-q-1}\frac{1}{m}\sum_{\zeta^m=1} \zeta^{-q}= \begin{cases}
        x^{-q-1}& m\mid q \\
        0 & m\not\mid q
    \end{cases}.
    \end{gather}
    \item Since $\mathcal{H}$ is obtained from $H^{(0,2)}$ by replacing $x_2^{-1-q} \to\begin{cases}
            x_2^{-1-q}&m\mid q \\ 0 & m\not| q.\end{cases}$, this follows directly from \textit{2.}
    \item Obtained by substituting the specialized \ref{eq:bergman} into the right-hand side of \textit{3} and simple computations that we omit.
\end{enumerate}
\end{proof}

From now on, for ease of use of the results of~\cite{bouttier2021}, we introduce slightly different versions of~{$A$} and~{$M$}. We set
    \begin{gather}
        \mathcal{A} = wA^{(0,1)} = \sum_{r\geq 0}A^{(0,1)}_{r} w^{r+1},  \\
        \mathcal{M} = M^{(0,1)}-H^{(0,1)}= \sum_{p\geq0, r\geq 1} M^{(0,1)}_{p,r} w^r x^{-1-p}. \notag
    \end{gather}
    
    In this notation, the functional equation of~\cite{bouttier2021} (formula~(26) of the~cited paper) reads
    \begin{gather}\label{pf:disk_peeling}
        \mathcal{M}(x,w)= w\left[\sum t_l x^l \left(\mathcal{M}(x,w) + H^{(0,1)}(x)\right)\right]_{x^{<0}} + \\
        + wx H^{(0,1)}(x)\left(\mathcal{M} + H^{(0,1)}\right) + \mathcal{A}(w) \mathcal{M}(x,w). \notag
    \end{gather}
    
    Specializing~\eqref{pf:disk_peeling} to~{$m$}-constellations, we get 
    \begin{gather}\label{pf:disk_peeling_m}
        \mathcal{M}(x,w)= w\left[x^m \left(\mathcal{M}(x,w) + H^{(0,1)}(x)\right)\right]_{x^{<0}} + \\
        + wx H^{(0,1)}(x)\left(\mathcal{M} + H^{(0,1)}\right) + \mathcal{A}(w) \mathcal{M}(x,w),\notag
    \end{gather}
    where the generating functions are those specialized for~{$m$}-constellations.

    Taking a look at the~{$x^{<0}$} coefficient in~\eqref{pf:disk_peeling_m}, we get:
    \begin{gather} \label{pf:coef_extract}
    \left[x^m (\mathcal{M} + H^{(0,1)})\right]_{x^{<0}}= x^m (\mathcal{M} + H^{(0,1)}) - x^{m-1} \left[\mathcal{M} + H^{(0,1)}\right]_{x^{-1}}, \\
    \left[\mathcal{M} + H^{(0,1)}\right]_{x^{-1}} = \frac{\mathcal{A}}{w}.
    \end{gather}
    So the recursion~\eqref{pf:disk_peeling_m}, once rearranged, reads
    \begin{gather} \label{pf:01eq_exctracted}
        \left(1 -\mathcal{A} - wx(x^{m-1} + H^{(0,1)})\right)\left(\mathcal{M}+H^{(0,1)}\right) = \\
        = H^{(0,1)} - \mathcal{A}\cdot\left(x^{m-1}+H^{(0,1)}\right). \notag
    \end{gather}

    By Proposition~3 of~\cite{bouttier2021} we have~{$x^{m-1} + H^{(0,1)(x)}=Y(x)$}, and by Theorem~1 of~\cite{bouttier2021}, we have~{$\mathcal{A}(w(s))= 1-\frac{x(z(s))^{m-1}}{y(z(s))}$}. So, in the~{$s$}-coordinate, we have
    \begin{equation}\label{pf:3}
    1-\mathcal{A} - wx(x^{m-1} + H^{(0,1)}) = 0.
    \end{equation}
More precisely, the equality holds when substituting $w=w(s), x=x(s)$.
    
We are now ready to prove the next lemma we need.
\begin{lemma}\label{lem:step2}
    When specialized to $m$-constellations and written in the $s$-coordinate we have
    \begin{equation}
    P^{(0,2)}(w_2,x_1) = \lambda\, \mathcal{H}(x_1, x_2) + \mu\, Q(x_1, x_2, w_2),
    \end{equation}
    where
    \begin{gather}
    Q(x_1,x_2,w) = \sum_{k\geq 1} mk\, x_1^{-1-mk}\left[x_2^{mk}M^{(0,1)}\right]_{x_2^{<0}}, \\
        \lambda = \frac{x_2 \left(x_2^{m-1}+\mathcal{M}(w_2, x_2)+2H^{(0,1)}(x_2)\right)}{x_2^{m-1} - \mathcal{M}(w_2, x_2)}, \\
        \mu = \frac{1}{x_2^{m-1} - \mathcal{M}(w_2, x_2)}.
    \end{gather}
\end{lemma}

\begin{proof}
        We start with \eqref{pf:disk_peeling} (with arguments $x_2, w$) and apply $L=L_{x_1}=\sum_{k\geq 1} mkx_1^{-1-mk}\frac{\partial}{\partial t_{mk}}$ to both sides:
        \begin{gather}
    L\mathcal{M} = \\
    = w \sum_{k\geq 1} mk x_1^{-mk-1}\left[x_2^{mk} \left(\mathcal{M} + H^{(0,1)}\right)\right]_{x_2^{<0}}+w\left[x_2^m \left(L\mathcal{M} + \mathcal{H}\right)\right]_{x_2^{<0}}+\notag \\
    + wx_2 \mathcal{H} (\mathcal{M} + H^{(0,1)})+wx_2H^{(0,1)} \left(L\mathcal{M}+ \mathcal{H}\right) + \notag\\
    +w P^{(0,2)}+ A^{(0,1)} L \mathcal{M}. \notag
\end{gather}

The first summand on the right-hand side is~{$Q$} from the lemma statement. The second, in a similar way as \eqref{pf:coef_extract}, reduces to
\begin{equation}
wx_2^m \left(L\mathcal{M} + \mathcal{H}\right) -wx_2^{m-1}P^{(0,2)}.
\end{equation}

Rearranging, we get
\begin{gather}
    L\mathcal{M}\left(1-\mathcal{A} - wx_2(x_2^{m-1}+H^{(0,1)})\right)= \\
    = Q+wx_2\mathcal{H}(\mathcal{M}+H^{(0,1)}) + wx_2 H^{(0,1)}\mathcal{H}+wP^{(0,2)}-wx_2^{m-1}P^{(0,2)} \notag.
\end{gather}
By~\eqref{pf:3} the left-hand side is zero and we immediately get the desired expression for~{$P^{(0,2)}$}.
\end{proof}

Rewriting~\eqref{pf:01eq_exctracted} back in terms of~{$M^{(0,1)} = \mathcal{M} + H^{(0,1)}$} and using $x^{m-1}+H^{(0,1)}=Y$:
\begin{equation}
(1-\mathcal{A} - wxY)M^{(0,1)} = H^{(0,1)}-\mathcal{A} Y.
\end{equation}
Applying {$\partial_w$} to the equality and, keeping in mind \eqref{eq:on_a_curve}, we get
\begin{equation}
M^{(0,1)} = \frac{Y \partial_w \mathcal{A}}{\partial_w \mathcal{A} + xY}.
\end{equation}
From here we get a following expression for the denominator of~{$\lambda$}:
\begin{equation}\label{pf:lam_denom}
    x_2^{m-1} - (M^{(0,1)}-H^{(0,1)}) = Y-M^{(0,1)}= \frac{xY^2}{\partial_w \mathcal{A} + xY}.
\end{equation}

Similarly for the numerator:
\begin{equation}\label{pf:lam_num}
x_2^{m-1}+M^{(0,1)}+H^{(0,1)}=Y+M^{(0,1)}= \frac{Y(xY +2\partial_w\mathcal{A})}{\partial_w \mathcal{A} + xY}.
\end{equation}

Using Theorem~1 of~\cite{bouttier2021}, we have, like in discussion of \eqref{pf:3}:
\begin{equation}
        \mathcal{A}(w(s)) = 1-\frac{\chi(s)}{\xi(s)} = 1-w(s)\xi(s).
\end{equation}

So, at least when working in the~{$s$} coordinate,
\begin{equation}
    \partial_w \mathcal{A} = -\xi - w \frac{d\xi}{dw}.
\end{equation}

\begin{lemma}\label{lem:step2.5}
    In the notation of  Theorem~\ref{thm:a02m_solution} we have
    \begin{gather}
        x_1 M^{(0,1)}(w_2,x_1) = \frac{\chi_1 - \chi_2}{w_2(\xi_1 - \xi_2)} - \xi_2; \\
        x_2 M^{(0,1)}(w_2,x_2) = \frac{1}{w_2} \frac{d\chi_2}{d\xi_2} - \xi_2.
    \end{gather}
    The second identity should be understood as a regular value as $s_1 \to s_2$ with a fixed $s_2$. 
        \end{lemma}
    \begin{proof}
        Start from the functional equation \eqref{pf:disk_peeling_m}, rewrite it for~{$M^{(0,1)}$} and in the variables {$x_1$} and~{$w_2$}:
\begin{gather}\label{pf:4}
\left(1-\mathcal{A}(w_2) - w_2x_1Y(x_1)\right)M^{(0,1)}(w_2,x_1)= \\
= H^{(0,1)}(x_1) - \mathcal{A}(w_2)Y(x_1).\notag
\end{gather}
Note that, since~{$\mathcal{A}(w_2) = 1- w_2 \xi_2$} and~{$\xi = xY$}, we have
\begin{equation}\label{pf:5}
\left(1-\mathcal{A}(w_2) - w_2x_1Y(x_1)\right) = w_2 \xi_2 - w_2\xi_1=w_2(\xi_2 - \xi_1).
\end{equation}
Similarly 
\begin{equation}\label{pf:6}
    x_1 \left(H^{(0,1)}(x_1) - \mathcal{A}(w_2)Y(x_1)\right)= w_2 \xi_1 \xi_2 - \chi_1.
\end{equation}
Substituting~\eqref{pf:5} and~\eqref{pf:6} into~\eqref{pf:4}:
\begin{gather}
    x_1 M^{(0,1)}(w_2,x_1) = \frac{-w_2 \xi_1 \xi_2 + \chi_1}{w_2(\xi_1 - \xi_2)} = \mat{\chi_2=w_2 \xi_2^2}= \\
    = \frac{\chi_1 - \chi_2 + w_2 \xi_2^2 - w_2 \xi_1\xi_2}{w_2 (\xi_1 - \xi_2)}= \frac{\chi_1 - \chi_2}{w_2(\xi_1 - \xi_2)} - \xi_2. \notag
\end{gather}
To get the second equality, follow these steps in the $x_2, w_2$ variables and take a limit. 
    \end{proof}

 Next, by direct computation that we omit (starting from the right-hand side and arriving at definition of $Q$), we check that
\begin{equation}\label{pf:Q_trick}
    Q(x_1, x_2, w) = -\partial_{x_2}\left(\frac{x_1^m M^{(0,1)}(x_2, w) - x_2^{m-1}x_1 M^{(0,1)}(x_1, w)}{x_1^m - x_2^m}\right).
\end{equation}
Substituting the results of Lemma~\ref{lem:step2.5} into \eqref{pf:Q_trick} we can check that 
\begin{equation}\label{pf:Q_summand}
    \frac{Q(x_1, x_2, w_2)}{x_2^{m-1}- \mathcal{M}(w_2,x_2)}dx_1 \frac{dw_2}{w_2} = -B_{\chi} + B_{\xi}. 
\end{equation}
    
Now we are ready to prove another lemma.

\begin{lemma}\label{lem:step3}
    When specialized to $m$-constellations and written in the $s$-coordinate we have
    \begin{equation}
        P^{(0,2)}(w_2,x_1)\, dx_1 \frac{dw_2}{w_2} = B_{\xi} - B_s.
    \end{equation}
\end{lemma}

\begin{proof}

    Putting~\eqref{pf:lam_denom} and~\eqref{pf:lam_num} together, we get:
    \begin{gather} 
        \lambda = \frac{\partial_w \mathcal{A} + xy}{y^2}\cdot \frac{y(xy +2\partial_w \mathcal{A})}{\partial_w \mathcal{A}+xy}= \\ 
        \\ = \frac{xy + 2\partial_w \mathcal{A}}{y} = \frac{\xi + 2(-\xi - w\frac{d\xi}{dw})}{\xi}\cdot x = -x-2x\frac{w}{\xi}\frac{d\xi}{dw}. 
    \end{gather}

    To be able to transform the last expression we note that differentiating $w=\frac{x^m}{\xi^2}$ by $w$ gives
    \begin{gather}
        1= \frac{mw}{x} \frac{dx}{dw} - \frac{2w}{\xi} \frac{d\xi}{dw}.
    \end{gather}
    Multiplying my $x$ and substituting into the previous expression we get
    \begin{gather}
        \lambda = -x-2x\frac{w}{\xi}\frac{d\xi}{dw} = -mw \frac{dx}{dw}
    \end{gather}

    Recalling Lemma~\ref{lem:step1}, we see that~{$\lambda \mathcal{H} dx_1 \frac{dw_2}{w_2}=-B_s + B_{\chi}$}.
    
    Adding the already computed second summand, \eqref{pf:Q_summand}, we have proven the lemma.
\end{proof}

\subsubsection{Solving the equation}

Theorem~\ref{thm:FE_g0n2} provides the functional equation for~{$M^{(0,2)}$}, and we specialize it to~{$m$}-constellations. By Lemma~\ref{lem:divisibility} and
\begin{equation}
\left[x^m M^{(0,2)}\right]_{x^{<0}} = x^mM^{(0,2)} - x^{m-1}A^{(0,2)},
\end{equation}
we get the specialization
\begin{gather}\label{eq:solving}
K\,M^{(0,2)} = P^{(0,2)} - w_1x_1^{m-1}A^{(0,2)} - w_1H^{(0,1)}A^{(0,2)}
+ w_1x_1P^{(0,2)}M^{(0,1)} \\
 - w_1P^{(0,2)}A^{(0,1)} + w_1A^{(0,2)}M^{(0,1)} + \mathcal{F}, \notag
\end{gather}
where
\begin{equation}
K = 1 - w_1A^{(0,1)}(w_1) - w_1x_1Y(x),\quad Y(x) = x_1^{m-1} + H^{(0,1)}(x)
\end{equation}
is the same kernel as in the one-boundary problem of~\cite{bouttier2021}. In other words, {$K(x, w(x)) = 0$} and where~{$\mathcal{F} = \mathrm{Term\ f}$} of the equation.

Setting~{$w=w(x)$} in~\eqref{eq:solving}, so that the left-hand side is zero we get 
\begin{gather}
    0= P^{(0,2)}(1+w_1 x_1 M^{(0,1)}-w_1 A^{(0,1)}) - \\
    - w_1A^{(0,2)}(x_1^{m-1} +H^{(0,1)} -  M^{(0,1)}) + \mathcal{F}, \notag
\end{gather}

Rearranging, we get
\begin{gather}
    A^{(0,2)}(w_1, w_2) = \frac{P^{(0,2)}(w_2,x_1)(1+w_1x_1 M^{(0,1)}(w_1,x_1) - w_1 A^{(0,1)}(w_1))+ \mathcal{F}}{w_1(x_1^{m-1 }-\mathcal{M}(w_1,x_1))}.
\end{gather}

We shall consider parts for the above expression for $A^{(0,2)}$ one by one.

In the numerator:
\begin{gather}
    1+w_1 x_1 M^{(0,1)}(w_1, x_1) - \mathcal{A}(w_1)= 1+\frac{w_1}{w_1} \frac{d\chi_1}{d\xi_1} - w_1\xi_1 -1 +w_1 \xi_1 = \frac{d\chi_1}{d\xi_2}.
\end{gather}
Together with Lemma~\ref{lem:step3} this gives
\begin{gather}
    P^{(0,2)}(w_2,x_1)\frac{dx_1}{d\xi_1}=(B_\xi - B_s) \frac{w_2}{dx_1 dw_2}\frac{dx_1}{d\xi_1}=(B_\xi - B_s)\frac{w_2}{d\xi_1 dw_2}.
\end{gather}
In the denominator:
\begin{gather}\label{eq:denominator}
    w_1 (x_1^{m-1} - \mathcal{M}(w_1,x_1))=w_1(Y(x_1) - M^{(0,1)})=\\
    = w_1 \left(\frac{\xi_1}{x_1} - \frac{1}{x_1}\left(\frac{1}{w_1}\frac{d\chi_1}{d\xi_1} - \xi_1\right)\right) =-\frac{1}{x_1}\frac{d\chi_1}{d\xi _1} + 2\frac{w_1\xi_1}{x_1}=-\frac{\xi_1^2}{x_1}\frac{dw_1}{d\xi_1}.\notag
\end{gather}
So the first summand becomes.
\begin{gather}\label{pf:pt1}
    (B_\xi - B_s) \frac{w_2}{dx_1 dw_2}=m(B_\xi - B_s) \frac{1}{ds_1 ds_2}.
\end{gather}

Again in the numerator we have
\begin{gather}
    \mathcal{F}=\frac{w_1^2 w_2}{(w_1-w_2)^2}\,
            \left(M^{(0,1)}(w_1,x_1)-M^{(0,1)}(w_2,x_1)\right)
        - \frac{w_1 w_2^{\,2}}{w_1-w_2}\,
            \frac{\partial M^{(0,1)}}{\partial w_2}(x_1,w_2).
\end{gather}

From Lemma~\ref{lem:step2.5} we obtain further expressions for $M^{(0,1)}(w_2,x_1)$:
\begin{gather}\label{pf:m01xi}
    x_1 M^{(0,1)}(w_2,x_1) = \frac{\chi_1 - \chi_2}{w_2 (\xi_1-\xi_2)} - \xi_2 = \mat{\chi = w\xi^2} = \\
    =\frac{\chi_1 \xi_1^2-w_2 \xi_1 \xi_2 + w_2 \xi_2^2 - w_2 \xi_1\xi_2}{w_2 (\xi_1 - \xi_2)}= \xi_1 + \frac{\xi_1^2 (w_1 - w_2)}{w_2(\xi_1 - \xi_2)}.\notag
\end{gather}
Same for $M^{(0,1)}(w_1,x_1)$:
\begin{gather}
    x_1 M^{(0,1)}(w_1,x_1) = \frac{1}{w_1} \frac{d\chi_1}{d\xi_1} - \xi_1 = \frac{1}{w_1}\left(\frac{\xi_1^2 dw_1}{d\xi_1} + 2w_1\xi_1\right) - \xi_1= \frac{\xi_1^2}{w_1}\frac{dw_1}{d\xi_1} + \xi_1
\end{gather}

Differentiating \eqref{pf:m01xi} w.r.t. $w_2$ we get the other part of $\mathcal{F}$:
\begin{gather}
    \frac{\partial M^{(0,1)}(w_2,x_1)}{\partial w_2} = \frac{\xi_1^2}{x_1} \left( -\frac{w_1}{w_2^2}\frac{1}{\xi_1 - \xi_2} - \frac{w_1 - w_2}{w_2}\frac{1}{(\xi_1-\xi_2)^2} \frac{d\xi_2}{dw_2} \right).
\end{gather}
All together: 
\begin{gather}
    x_1\mathcal{F} = w_1 w_2 \xi_1^2 \left(\frac{1}{(w_1 - w_2)^2} \frac{dw_1}{d\xi_1} + \frac{1}{ (\xi_1 - \xi_2)^2} \frac{d\xi_2}{dw_2}\right). 
\end{gather}

Dividing by \eqref{eq:denominator} we obtain
\begin{equation} \label{pf:pt2}
    \frac{\mathcal{F}}{w_1 (x_1^{m-1}-\mathcal{M}(w_1,x_1))}=-\frac{x_1 \mathcal{F}}{w_1 w_2 \xi_1^2} d\xi_1 dw_2 = B_{\xi}-B_w.
\end{equation}

Putting together \eqref{pf:pt1} and \eqref{pf:pt2} we get

\begin{equation}
A^{(0,2)}(w_1,w_2)\frac{dw_1dw_2}{w_1w_2} = - m (B_{\xi} - B_s) + (B_{\xi}-B_w),
\end{equation}
which is equivalent the theorem statement and the theorem is proven.

\printbibliography

@book{eynard2016counting,
  title={Counting surfaces},
  author={Eynard, Bertrand},
  year={2016},
  publisher={Springer}
}

@misc{bouttier2021,
      title={Enumeration of planar constellations with an alternating boundary}, 
      author={Jérémie Bouttier and Ariane Carrance},
      year={2021},
      eprint={2012.14258},
      archivePrefix={arXiv},
      primaryClass={math.CO},
      url={https://arxiv.org/abs/2012.14258}, 
}

@misc{bouttier2026,
      title={Enumeration of plane hypermaps with a mixed boundary I}, 
      author={Jérémie Bouttier and Bertrand Eynard and Thomas Lejeune},
      year={2026},
      eprint={2608.19947},
      archivePrefix={arXiv},
      primaryClass={math.CO},
      url={https://arxiv.org/abs/2608.19947}, 
}

@misc{baillard2026,
      title={Enumeration of general planar hypermaps with an alternating boundary}, 
      author={Valentin Baillard and Ariane Carrance and Bertrand Eynard},
      year={2026},
      eprint={2603.28571},
      archivePrefix={arXiv},
      primaryClass={math.CO},
      url={https://arxiv.org/abs/2603.28571}, 
}

@article{Kazarian_2015,
   title={Virasoro Constraints and Topological Recursion for Grothendieck’s Dessin Counting},
   volume={105},
   ISSN={1573-0530},
   url={http://dx.doi.org/10.1007/s11005-015-0771-0},
   DOI={10.1007/s11005-015-0771-0},
   number={8},
   journal={Letters in Mathematical Physics},
   publisher={Springer Science and Business Media LLC},
   author={Kazarian, Maxim and Zograf, Peter},
   year={2015},
   month=June, pages={1057–1084} }

\end{document}